\documentclass[a4paper, 12pt, twoside]{amsart}
\usepackage{amsfonts,amssymb,amsmath,amsthm}
\usepackage{graphicx}
\usepackage[all]{xy}
\usepackage{multirow} 
\usepackage{psfrag,xmpmulti,amscd,color,pstricks, import}
 \usepackage{hyperref}
 \usepackage{tikz}

 \divide\oddsidemargin by 2
\makeatletter
\@namedef{subjclassname@2020}{%
  \textup{2020} Mathematics Subject Classification}
\makeatother

\newtheorem{thm}{Theorem}[section]
\newtheorem{cor}[thm]{Corollary}
\newtheorem{lem}[thm]{Lemma}
\newtheorem{prop}[thm]{Proposition}

\newtheorem*{mainthm}{Main Theorem}
\theoremstyle{definition}
\newtheorem{defn}[thm]{Definition}
\newtheorem{rem}[thm]{Remark}
\newtheorem*{rem*}{Remark}

\numberwithin{equation}{section}

\definecolor{OrangeRed}{cmyk}{0,0.6,1,0}            
\definecolor{DarkBlue}{cmyk}{1,1,0,0.20}
\definecolor{DarkGreen}{cmyk}{1,0,0.6,0.2}
\definecolor{myblue}{rgb}{0.66,0.78,1.00}
\definecolor{Violet}{cmyk}{0.79,0.88,0,0}
\definecolor{Lavender}{cmyk}{0,0.48,0,0}

\renewcommand{\Im}{\operatorname{Im}}
\renewcommand{\Re}{\operatorname{Re}}
\newcommand{\re}{\operatorname{Re}}

\newcommand{\Hol}{\operatorname{Hol}}

\newcommand{\C}{{\mathbb C}}

\newcommand{\D}{{\mathbb D}}

\newcommand{\N}{{\mathbb N}}

\renewcommand{\P}{{\mathbb P}}
\newcommand{\R}{{\mathbb R}}

\newcommand{\ra}{\rightarrow}

\renewcommand{\epsilon}{\varepsilon}
\renewcommand{\phi}{\varphi}
\renewcommand{\emptyset}{\varnothing}

\title[Non-bulging Baker domains for transcendental skew products]{Non-bulging Baker domains for transcendental skew products}
 
\vspace{5cm}
\author{Anna Miriam Benini}
\address{Dipartimento di Scienze Matematiche Fisiche e Informatiche, Universit\`a di Parma, Parma,  IT}
\email{ambenini@gmail.com}

\author{Tom Potthink}
\address{Tom Potthink, Univ. Bordeaux, CNRS, Bordeaux INP, IMB, UMR 5251, F-33400 Talence, France}
\email{tom.potthink@math.u-bordeaux.fr}

\author{Jasmin Raissy}

\address{Jasmin Raissy, Univ. Bordeaux, CNRS, Bordeaux INP, IMB, UMR 5251, F-33400 Talence, France \& Institut Universitaire de France (IUF)}
\email{jasmin.raissy@math.u-bordeaux.fr}

\thanks{This material is based upon work supported by the National Science Foundation under Grant No. 1440140, while the first and third  authors were in residence at the Mathematical Sciences Research Institute in Berkeley, California, during the spring semester of 2022.  Partially supported by the Institut Universitaire de France (IUF) and 
the French Italian University and Campus France through the PHC Galileo program, under the project ``From rational to transcendental: complex dynamics and parameter spaces''. The first author was also partially supported by the Indam group GNAMPA}

\begin{document}

\maketitle  
\begin{abstract}
In this paper we show that Baker domains of transcendental skew products can either bulge or not, depending on the higher order terms. This is in contrast to polynomial skew products where all Fatou components with bounded orbits of an invariant attracting fiber do bulge.
\end{abstract}

\section{Introduction}
A skew product is a holomorphic map $F:\C^2\ra\C^2$ of the form
\begin{equation}\label{eq:skew product}
F(z,w)=(f(z,w),g(w)).
\end{equation}
  Notice that $f:\C^2\to\C$ and $g:\C\to \C$, and that the dynamics in the second component is one-dimensional. 
  
  As innocent as they may appear at first sight, their dynamics is far from being trivial. For instance, they provided the first known example of wandering domains for polynomial endomorphisms of $\C^2$ \cite{ABDPR}. 
  
  It is natural to ask to what extent the one-dimensional dynamics in the second component influences the global two-dimensional dynamics. The simplest question is to investigate the relation between the set of normality for the one-dimensional map $g$ and the one for  the two-dimensional map $F$. 
 In general, it is difficult to pinpoint the relation between normality of iterates  of a map $F=(f_1,f_2):\C^2\ra\C^2$ and normality of iterates of its components $f_1, f_2$, in however way one wants to make sense of it. The definition of normality itself  for maps of $\C^2$ is not choice-free, since normality requires the target space to be compact. While it is commonly agreed  that the correct compactification of $\C$ is the Riemann sphere $\widehat{\C}$, the one point compactification of $\C^2$ is not a complex manifold, and there are several natural compactifications of $\C^2$ that are complex manifolds (see \cite{Morrow} and \cite{Brenton} for example). Here we will choose as compactification the projective space (see the discussion in \cite{ABFP2019} and Definition~\ref{def:normality}). 
  
  If we assume that $w_0\in\C$ is an attracting fixed point for $g(w)$, that is $g(w_0)=w_0$ and $|g'(w_0)|<1$, then the fiber $\{w=w_0\}$ is an invariant curve in $\C^2$ and  $f_{w_0}:=f(\cdot, w_0):\C\to\C$ is a one-dimensional map acting on $\{w=w_0\}$.  Hence, it has its own Fatou components coming from the one-dimensional theory. Since the dynamics of $F$ near $\{w=w_0\}$ is attracting in the $w$ direction, it is natural to expect that each Fatou component for $f_{w_0}$ is contained in a Fatou component for $F$. Without loss of generality, for the remainder of the paper we will assume that $w_0=0$. 
     
   \begin{defn}[Bulging]\label{def:bulging}
  A Fatou component $U$ for $f_{0}$ \emph{bulges} if there exists a Fatou component $\Omega$ for $F$ such that $U\subset\Omega\cap\{w=0\}$.
  \end{defn} 
  
  By Sullivan's No Wandering Domains Theorem \cite{Sullivan1985} it is known that all Fatou components of a one-dimensional polynomial are periodic or pre-periodic. Moreover, any invariant Fatou component of such polynomials is either the basin of an attracting fixed point, or a parabolic basin, or a Siegel disk. 
  Therefore the classification of Fatou components for one-dimensional polynomial maps  implies that, except for the basin of infinity, iterates are uniformly bounded, which with some work allows to conclude that each Fatou component with bounded orbits for $f_{0}$ \emph{bulges} to a Fatou component of $F$. 
  
   This fact is summarized in the following result obtained by Lilov \cite{Lilov}.

  \begin{thm}[Bulging for polynomial skew products, \cite{Lilov}]
  If $F$ is a polynomial skew product and $0$ is an attracting fixed point for $g$, then all Fatou components of $f_{0}$ with bounded orbits bulge. 
  \end{thm}
   
   More precisely, Lilov proved in \cite[Sect. 3.2.1]{Lilov} the bulging of attracting basins of fixed points different from infinity,  of parabolic basins, and of Siegel disks for $f_{0}$ when $0$ is an attracting fixed point for $g$.
   One can also relax the hypothesis that $0$ is an attracting fixed point for $g$ and make it the center of a Siegel disk; in this case, the picture becomes more intricate \cite{PetersRaissy}. 
  
   In fact, more can be asked: Is $\Omega \cap \{w=0\}$ \emph{equal} to $U$? And can there be \emph{other} Fatou components for $F$ which do not intersect $\{w=0\}$, yet intersect every  neighbourhood thereof? Several answers when $f(z,w)$ is a polynomial can be found in the survey \cite{Raissy2017} and the references therein. 

 When $f_0$ is transcendental rather than polynomial, two new types of Fatou components appear: \emph{wandering domains} (components of the Fatou set which are not pre-periodic) and \emph{Baker domains} (periodic Fatou components on which the iterates converge to the essential singularity at infinity). One needs to resist the temptation to consider Baker domains as  an analogue of parabolic domains for which the fixed point has been moved to infinity (see the  classification of Baker domains  \cite{FagellaHenriksen}, and compare to  \cite{Cowen}). On wandering domains, the iterates of the map can converge to infinity, have both bounded and unbounded subsequences, or possibly be bounded (there are no known examples of this last case occurring, but it has not been ruled out either).
  
   For the types of Fatou components for $f_{0}$ which also occur in the polynomial case, everything works as expected and follows from the proofs of Lilov in \cite{Lilov}.
  
   \begin{thm}[{\cite[Sect. 3.2.1]{Lilov}}]
  Let $F$ be  a transcendental skew product with an invariant attracting fiber $\{w=0\}$, such that $f_{0}$ has an invariant Fatou component $U$ with bounded orbits.  Then $U$ bulges.
  \end{thm}
  
  In this paper we study the bulging of Baker domains, and we show that the fact that iterates tend to infinity makes it possible for Baker domains to either bulge or not depending on the choice of the perturbation. Our main result is the following theorem (see Section~\ref{sec:Baker} for the fact that $z\ra z+1+e^{-z}$ has a Baker domain, and Section~\ref{bulging} for the definition of order  in the first variable).
  
 \begin{mainthm}
Let $f_{0}(z)=z+1+e^{-z}$, let $B$ be its Baker domain, and let $F(z,w)$ be a skew product of the form 
$$
F(z,w)=(f_{0}(z)+w h (z,w),g(w)),
$$
 for which $g$ has an attracting fixed point at $0$. 
Then \begin{itemize}
\item If  the order  of $h$ in the first variable is less than $1$, then $B$ bulges. 
\item If  $g'(0)=0$, and  the order  of $h$ in the first variable is finite, then $B$ bulges. 
\end{itemize}
On the other hand, there exists a transcendental entire function $h(z)$ such that $B$ does not bulge.
  \end{mainthm}
  In particular, if  $ h(z,w)$ does not depend on $z$, then the order in the first variable is zero, and the  Baker domain $B$ bulges. The non-bulging case is first of its kind in the attracting setting. Moreover, the proof in the bulging case relies on the analysis and control of the growth given by the perturbation $w h(z,w)$ in the first coordinate, which is conceptually different from the strategy used by Lilov in \cite{Lilov}. 
    
Bulging of  wandering domains, the last type of Fatou components to consider,  will be treated in a forthcoming paper by the second and last authors \cite{TJ}. 

\smallskip
\subsection*{Structure of the paper} In section~\ref{preliminaries}, we give some preliminaries on the definition of normality in $\C^2$ and on Baker domains. In section~\ref{bulging}, we construct examples of bulging Baker domains. In section~\ref{non-bulging}, we construct the example of a non bulging Baker domain.

\subsection*{Notation} We denote by $\C$ the complex plane, by $\P^2$ the complex projective space of dimension 2, and by $\D(a,r)$ the Euclidean disk of radius $r$ centered at $a$. Given $F:\C^2\ra\C^2$ and a point $P=(z,w)$, we write its orbit as  $(z_n, w_n):=F^n(z,w)$.

\subsection*{Acknowledgments} This project grew up from discussions with Alberto Saracco on the dynamics of transcendental skew products and we thank him for his input which led to the present work.
The first and last author also thank the MSRI and NSF for their support during the program ``Complex Dynamics: from special families to natural generalizations in one and several variables'' which took place during the spring semester of 2022 and where part of this work was initiated. This project was supported in part with funding from the Institut Universitaire de France (IUF) and 
the French Italian University and Campus France through the PHC Galileo program, under the project ``From rational to transcendental: complex dynamics and parameter spaces''. The first author was also partially supported by the Indam group GNAMPA.

\section{Preliminaries: Normality, Baker domains, Absorbing domains}\label{preliminaries}

\subsection{Normality}

We  compactify $\C^2$ as $\P^2$ by adding a complex line at infinity, which entails  the following definition of normality \cite{ABFP2019}.

\begin{defn}\label{def:normality}
Let $X\subset \C^2$ be a domain. 
A family $\mathcal{F}\subset {\rm Hol}(X,\C^2)$ is  \emph{normal}  if for every sequence $(F_n)_{n\ge 0}\subset\mathcal{F}$ there exists a subsequence $(F_{n_k})_{n_k\ge 0}$ converging uniformly on compact subsets to $G\in {\rm Hol}(X,\P^2)$.
In other words, $\mathcal{F}$ is  relatively compact in ${\rm Hol}(X,\P^2)$.
\end{defn}

  We notice that for polynomial automorphisms, this definition of normality  is equivalent to the one induced by the one-point compactification of $\C^2$.  Indeed by \cite{FM89}, it is enough to look at H\'enon maps, and for the latter, the filtration in  \cite{BS91} implies that every escaping orbit converges to the point $[1:0:0] \in \P^2$.
  
  The following result shows that, for a given Fatou component, all limit functions are either holomorphic maps in $\Hol(X, \C^2)$, or take values entirely in the compactifying part at infinity.

\begin{lem}[{\cite[Lemma 2.4]{ABFP2019}}]\label{Lemma:NormalityRelationsBakerVariant}
	Let $X\subset \C^2$ be a domain. If a family $\mathcal{F} \subset \Hol(X, \C^2)$ is normal, then either any limit function lies in $\Hol(X, \C^2)$, or the corresponding sequence diverges locally uniformly to infinity.
\end{lem}  
  
  \subsection{Baker domains and absorbing domains}\label{sec:Baker}
 Given a transcendental entire function $f:\C\ra\C$, a \emph{Baker domain} $B$ for $f$ is a periodic Fatou component on which orbits tend to infinity. The first example of Baker domains was given by Fatou \cite[Example I]{fatou19}, who proved that for the function $f(z)=z+1+e^{-z}$ orbits tend to infinity on the right half plane $\{z\in\C\mid\re z>0\}$.
 
One can show indirectly that Baker domains of entire functions are simply connected  \cite{BakerMCD}, hence biholomorphic to the unit disk $\D$ by the Riemann Mapping Theorem. 
 
 \begin{defn} An \emph{absorbing domain} for a Baker domain $B$ is an open, simply connected, forward invariant  set $D\subset B$ such that for every compact set $K\subset B$, there exists $n=n(K)\ge 0$ such that $f^{n}(K)\subset D$. 
 \end{defn}
 
It follows by results of Cowen that every Baker domain has absorbing domains (\cite{Cowen}; compare with  \cite{BFJK}).

We will use entire functions similar to Fatou's first example, namely of the form $f(z)=z+a+e^{-z}$ with $\re a>0$. These functions are sometimes called \emph{Fatou functions} (see for example \cite{Vasso}). The following fact is well known (the first property is elementary; for the second one, see e.g.  \cite[proof of Theorem 2.5]{KotusUrbanski05}). 

 \begin{lem}\label{lem:basic Baker 1d}
 Let  $f(z)=z+a+e^{-z}$ with $\re a>0$. Then  for any $R>|\log \re a|$,
 \begin{itemize}
 \item the right half plane $\{z\in\C\mid\re z>R\}$ is invariant under $f$ and contained in an invariant Baker domain $B_a$.
 \item the right half plane $\{z\in\C\mid\re z>R\}$ is an absorbing domain for $B_a$.
 \end{itemize}
 \end{lem}
%

We will also need the following simple lemma for Fatou's function $f(z)=z+1+e^{-z}$. We report here the proof for the sake of completeness.

\begin{lem}\label{LEMMA:OneDim}
    Consider the function $f(z) = z + 1 + e^{-z}$. Then for any $\delta>0$ there exists a real number $x_0 > \log(2)$ such that for all integers $k\ge 1$
        \begin{equation}\label{LEMMA:OneDim:EQ1}
            x_k := f^{k}(x_0) \in (x_0 + k - \delta, x_0 + k + \delta) \subset \R,
        \end{equation}
        and
        \begin{equation}\label{LEMMA:OneDim:EQ2}
            |f(z) - f(\tilde{z})| < 10\delta
        \end{equation}
        for all $z,\tilde{z}\in \overline{\D(x_0 + k, 4\delta)}$ and $k\in\N$.
\end{lem}

\begin{proof}

    Let $\delta > 0$ be arbitrary. 
    Since the series $\sum_{j=0}^{\infty} e^{-\frac{j}{2}}$ converges, we may choose $x_0 > \log 2 $ large enough such that
    \begin{equation}\label{PROOF:LEMMA:OneDim:EQ2}
        e^{-x_0}\sum_{j=0}^{\infty} e^{-\frac{j}{2}} = \sum_{j=0}^{\infty} e^{-x_0 -\frac{j}{2}} < \delta.
    \end{equation}
    and
    \begin{equation}\label{PROOF:LEMMA:OneDim:EQ3}
        e^{-x_0 + 4 \delta} < \delta.
    \end{equation}
    Therefore, by induction, we have that
    \begin{equation*}
        |x_0 + k - f^{k}(x_0)| \leq \sum_{j=0}^{k-1} e^{-x_0 -\frac{j}{2}}
    \end{equation*}
    for all integers $k\ge 1$, where empty sums are interpreted as $0$. Using \eqref{PROOF:LEMMA:OneDim:EQ2} and the fact that $f(x) \in \R$ for all $x \in \R$, we obtain \eqref{LEMMA:OneDim:EQ1}.


    Finally, for property \eqref{LEMMA:OneDim:EQ2}, let $z,\tilde{z}\in \overline{\D(x_0 + k, 4\delta)}$ and $k\in\N$. Then, using \eqref{PROOF:LEMMA:OneDim:EQ3}, we obtain
    \begin{equation*}
        \begin{split}
            |{f(z) - f(\tilde{z})}| &\leq |{z - \tilde{z}}| + |{e^{-z}}| + |{e^{-\tilde{z}}}| \\
            &\leq 8\delta + e^{-\Re(z)} + e^{-\Re(\tilde{z})} \\
            &\leq 8\delta + 2 e^{-x_0 - k + 4 \delta} 
            %
            \leq 8\delta + 2 e^{-x_0 + 4 \delta} 
            %
            %
            \leq 10\delta,
        \end{split}
    \end{equation*}
    which completes the proof.
\end{proof}
  
\section{Bulging Baker domains}\label{bulging}

In this section, we give conditions under which the Baker domain of a Fatou's function bulges. The following result shows that if an absorbing  domain bulges, so does the whole Baker domain.

\begin{prop}\label{lem:absorbing implies everything} 
Let $g\colon\C\to\C$ be an entire function with an attracting fixed point  at $0$. Let $F \colon \mathbb{C}^2 \to \mathbb{C}^2$ be a transcendental skew product of the form $F(z,w)=(f(z,w),g(w))$. Assume that $f_0$ has a Baker domain $B$ and let  $D$ be  an absorbing domain for $B$. If there is a Fatou component $\Omega$ for $F$ such that $D\subset\Omega\cap \{w=0\}$, then $B$ bulges to $\Omega$.
\end{prop} 

\begin{proof} Let $0 < \delta_g < 1$ be such that the disk $\D(0,\delta_g)$ lies entirely within the immediate attracting basin of $g$ at $0$, i.e., $|g(w)|\le|w|$ for all $x\in \D(0,\delta_g)$ and so $g(\D(0,\delta_g)) \subset \D(0,\delta_g)$. Moreover, up to considering an iterate, we may assume that the Baker domain $B$ is invariant.

Let $\Omega$ be the invariant Fatou component for $F$ containing the absorbing domain $D$, that is $D\subset\Omega$. It is enough to show that for every non-empty compact set $K\subset B$ we have $K\subset\Omega$.  
Up to enlarging $K$, since $B$ is connected we can assume that $K\cap D\neq\emptyset$. 
Since $D$ is an absorbing domain for $f_0$, there exists $N=N( K)\ge 0$ such that $f_0^N(K)\subset D$. Let $Q$ be a closed $\delta$-neighbourhood of $f_0^N(K)$ where $\delta>0$ is chosen small enough so that $Q\Subset D$.

 Since  $Q$ is compact and $D\subset\Omega\cap\{w=0\}$, there exists $0<\epsilon_Q<\delta_g$ such that $ Q\times\ D(0,\epsilon_Q)\Subset \Omega$.
By continuity of $F^N$, there is  $0<\epsilon_K<\epsilon_Q$ such that on $K\times \overline{\D(0, \epsilon_K)}$ we have $\|F(z,w)-F(z,0)\|<\delta/2$ .
In other words we can write 
\begin{equation}\label{eq:2.3}
F^N(z,w)=(z_N, w_N)=(f_0^N(z)+ q(z,w), g^N (w),)
\end{equation}
where $q(z,w)$ is a holomorphic function such that $|q(z,w)|<\delta/2$  on $K\times \overline{\D(0, \epsilon_K)}$.

We now claim that 
$$
F^N(K\times\overline{\D(0, \epsilon_K)})\subset Q\times\D(0, \epsilon_Q)\subset \Omega.
$$
Indeed, if $P=(z,w)\in K\times\overline{\D(0, \epsilon_K)}$, set $(z_N, w_N)= F^N(P)$. Since $|w|\le\epsilon_K<\epsilon_Q<\delta_g$ and $|g(w)|\le|w|$ on $\D(0,\delta_g)$, we also have that  $|w_N|\le|w|<\epsilon_Q$.  
For the first coordinate,  $|z_N-f_0^N(z)|=|q(z,w)|<\delta/2$, hence, since $f_0^N(z)\in f_0^N(K) $ and $Q$ is a $\delta$-neighbourhood of $f_0^N( K)$, we obtain $z_N\in Q$.

This implies that  $K\times \overline{\D(0, \epsilon_K)}$ is contained in a Fatou component $\Omega_K$ of $F$ (a priori, this could be a preimage of $\Omega$). However since $K\cap D\neq\emptyset$, $D\subset\Omega$, and Fatou components are connected, we have that $\Omega_K=\Omega$.

\end{proof}

We observe the following properties of bulging Baker domains. 

\begin{prop}\label{prop: properties Baker}
Let $g\colon\C\to\C$ be an entire function with an attracting fixed point  at $0$. Let $F \colon \mathbb{C}^2 \to \mathbb{C}^2$ be a transcendental skew product of the form $F(z,w)=(f(z,w),g(w))$. Assume that $f_0$ has a Baker domain $B$. If $B$ bulges to a Fatou component $\Omega$,  then  $F^n\ra[1:0:0]$ as $n\ra\infty$ on $\Omega$.  Moreover, if $B$ coincides with the Fatou set of $f_0$,  then  $\Omega\cap \{w=0\}=B$.
\end{prop}
\begin{proof}
 We first show that $F^n\ra[1:0:0]$ as $n\ra\infty$ on $\Omega$.  Let $P\in B\times \{0\}$. Clearly  $F^n(P)\ra[1:0:0]$.  Hence by Lemma 2.4 in \cite{ABFP2019}, for any   limit function $h$ on $\Omega$, we obtain that $h(\Omega)\subset\ell_\infty$.  Let $\Delta$ be a small polydisk centered at  $P$ and $Q=(z^Q,w^Q)\in \Delta$. Then $w^Q_n\ra0$ because   $0$ is attracting for $g$, yet $\|F^n(Q)\|\ra\infty$ because for any   limit function $h$ on $\Omega$, we have that $h(\Omega)\subset\ell_\infty$. So $|z^Q_n|\ra\infty$ for any $Q\in \Delta$, hence  $F^n\ra[1:0:0]$ on $\Delta $ as $n\ra\infty$. By the identity principle, we conclude that $F^n\ra[1:0:0]$ on $\Omega$ as $n\ra\infty$.
 
We now  show that if $P\in\Omega\cap \{w=0\}$, then $P$ belongs to the Fatou set of $f_0$. If $f_0$ has no other Fatou component,  this implies that $\Omega\cap\{w=0\}=B$. 

If $P=(z,0)\in\Omega$, then there exists an open neighbourhood $U\subset\Omega$ of $P$ such that  $F^n\ra[1:0:0]$ on $U$. In particular, $\pi_z\circ F^n\ra\infty$ on $U$, where  $\pi_z$ is the projection on the first coordinate. Thus the iterates of $f_0$ are  normal on $U\cap\{w=0\}$, concluding the proof.
\end{proof}

As a corollary of Propositions~\ref{lem:absorbing implies everything} and \ref{prop: properties Baker},  it is easy to construct examples of bulging Baker domains.

\begin{cor}\label{cor:bulging Baker domain}
Let $a \in \mathbb{C} $ with $\Re(a) > 0$, let $f$ be the entire function $f(z) = z + a + e^{-z}$, and let $g$ be an entire function with an attracting fixed point at $0$.  Let $h$ be a non-constant entire function such that $h(0)=0$, and consider  the skew product
\begin{equation*}
		 F(z,w) = (f(z) + h(w), g(w)).
	\end{equation*}
	Then  the Baker domain  $B$ of $f$ bulges to a Fatou component $\Omega$ for $F$ and moreover $\Omega\cap\{w=0\}=B$.
\end{cor}

\begin{proof}
Thanks to Proposition~\ref{lem:absorbing implies everything}, it suffices to prove that there is a bulging right half plane $\{z\in\C\mid \re z >R\}$ with $R>0$ sufficiently large, since by Lemma~\ref{lem:basic Baker 1d}, such half plane will be an absorbing domain for $f$. Using the fact that $h(0)= 0$ and the continuity of $h$ at $0$, an easy computation shows that for $R>0$ sufficiently large and $\epsilon>0$ sufficiently small, $F^n$ converges to $[1:0:0]$ as $n\ra\infty$ uniformly on the set 
$$
W_{\epsilon, R}:=\{z,w\in\C^2: \Re z>R,  |w|<\epsilon\},
$$
that is, the right half plane $\{z\in\C\mid \re z >R\}$ is bulging.
Since the only Fatou component for $f$ is  the Baker domain $B$, it follows from Proposition~\ref{prop: properties Baker} that  $\Omega\cap\{w=0\}=B$.
\end{proof}

It is reasonable to ask whether one can characterize the errors for which we have a bulging Baker domain, at least for the class of Fatou functions. 
For more complex error terms, it is necessary to quantify their growth. One approach uses the order, as presented in greater generality in \cite[Chapter 3]{Ronkin1974} and \cite{Ronkin1989}. Here, we restrict ourselves to the following definitions.

\begin{defn}[cf. {\cite[Sections 1.1 and 1.2]{Ronkin1989}}]
	Let $f : \mathbb{C}^2 \to \mathbb{C}$ be entire and non-constant. The \emph{maximum modulus} of $f$ at radii $r_1, r_2 >0$ is defined as
	\begin{equation*}
		M(r_1, r_2; f) := \max_{\lvert z \rvert \leq r_1, \lvert w \rvert \leq r_2} \lvert f(z, w) \rvert.
	\end{equation*}
	The \emph{order of $f$} is defined as
	\begin{equation*}
		\rho(f) := \limsup_{r\to\infty} \frac{\log^{+} \log M(r, r; f)}{\log r} \geq 0.
	\end{equation*}
	The \emph{order of $f$ in the $z$ variable, or first variable,} is defined as
	\begin{equation*}
		\rho_1(f) := \limsup_{r_1\to\infty} \frac{\log^{+} \log M(r_1, r_2; f)}{\log r_1} \geq 0
	\end{equation*}
	for some $r_2 > 0$.
\end{defn}

\begin{rem}\label{REMARK:OrderOfAFunction}
	\begin{enumerate}
		\item As discussed in \cite[Section 1.2]{Ronkin1989}, the definition of $\rho_1(f)$ is independent of the choice of $r_2 > 0$.
		\item Every polynomial has order $0$. Similarly, any function that is polynomial in $z$, i.e., is of the form
		\begin{equation*}
			(z,w) \mapsto \sum_{k=0}^{d} h_k(w) z^k
		\end{equation*}
		for $d \in \mathbb{N}_{0}$ and $h_k : \mathbb{C} \to \mathbb{C}$ holomorphic functions, has order $0$ in the $z$ variable.
		\item Simple examples of non-zero, finite order in one-dimension are given by entire transcendental functions of the form
		\begin{equation*}
			\exp_k : \mathbb{C} \to \mathbb{C}, z \mapsto e^{z^k},
		\end{equation*}
		as they satisfy
		\begin{equation*}
			\rho(\exp_k) = k
		\end{equation*}
		for $k \in \mathbb{N}_{0}$.
	\end{enumerate}
\end{rem}

%

\begin{thm}\label{Theorem:BakerProductErrorTEF}
	Let $g : \mathbb{C}\to\mathbb{C}$ be a non-constant entire function with an attracting fixed point at $0$, let $f : \mathbb{C} \to \mathbb{C}$ be a Fatou function of the form $f(z) = z + a + e^{-z}$ for some $a \in \mathbb{C} $ with $\Re(a) > 0$, and let $h : \mathbb{C}^2 \to \mathbb{C}$ be holomorphic with order $\rho_1(h)$ in the first variable. If 
	\begin{enumerate}
		\item $\rho_1(h) < 1$, or
		\item $0$ is super-attracting for $g$ and $\rho_1(h) < \infty$,
	\end{enumerate}
	then the Baker domain of $f$ bulges for the skew product $F : \mathbb{C}^2 \to \mathbb{C}^2$
	\begin{equation*}
		F(z,w) = (f(z) + wh(z,w), g(w)).
	\end{equation*}	
\end{thm}

\begin{rem}
	As a consequence, we have bulging for any error term which is polynomial in the $z$ variable, so Corollary~\ref{cor:bulging Baker domain} can also  be deduced as a corollary of Theorem~\ref{Theorem:BakerProductErrorTEF}.
\end{rem}

The rest of this section is devoted to the proof of Theorem~\ref{Theorem:BakerProductErrorTEF}.  We will need the following lemma. 

\begin{lem}\label{Lemma:Baker:RealPartToInftyVar2}
	Under the assumptions of Theorem \ref{Theorem:BakerProductErrorTEF}, there exists $L > 0$ large enough such that given any 
	\begin{equation*} 
		L \leq r < R 
	\end{equation*}
	there exists $\varepsilon >0 $ such that for all $k\in\mathbb{N}$, $z \in U_{r,R} := \lbrace z \in \mathbb{C} \mid \Re(z)>u, \lvert z \rvert < R \rbrace$, and $w \in \D(0,\varepsilon)$ we have
	\begin{enumerate}
		\item $\Re(z_k)>  r + \frac{k\Re(a)}{4} \ge L$, and \label{Lemma:Baker:RealPartToInftyVar2:ProofProp1}
		\item $ \lvert z_k \rvert < R + 2k \lvert a \rvert$, \label{Lemma:Baker:RealPartToInftyVar2:ProofProp2}
	\end{enumerate}
	where $(z_k, w_k) := F^k(z,w)$.
\end{lem}

\begin{proof}[Proof of Lemma {\ref{Lemma:Baker:RealPartToInftyVar2}}]
    Since $0$ is an attracting fixed point for $g$, that is $g(0) = 0$ and $\lvert g'(0) \rvert < 1$, there exists $0 < \alpha < 1$ and $\delta_g > 0$ such that for all $w \in D(0,\delta_g)$
	\begin{equation}
		\lvert g(w) \rvert \leq \alpha \lvert w \rvert.
	\end{equation}
Suppose first that  $\rho_1(h) < 1$.  Then we can choose $\tau > 0$ such that ${\rho_1(h) + \tau < 1}$. By definition, there exists $L > 0$ such that for all $\sigma \geq L$
	\begin{equation*}
		\log^+ \log M(\sigma, \delta_g; h) \leq (\rho_1(h) + \tau) \log \sigma 
	\end{equation*}
	and so
	\begin{equation*}
		M(\sigma, \delta_g; h) \leq e^{\sigma^{\rho_1(h) + \tau}}.
	\end{equation*}
	Remark that we could have chosen any other positive constant for the $w$ coordinate. 
	
	Furthermore, without loss of generality, we may assume $L > \log \frac{2}{\Re(a)}$, so that the right half plane $\{z\in\C\mid \re z>L\}$ is an absorbing domain for $f$. 
	Then, consider any radii satisfying $R>r\ge L$.
	We have
	\begin{equation*}
		e^{(R + 2k \lvert a \rvert)^{\rho_1(h) + \tau}} \lvert \alpha \rvert^k  = e^{(R + 2k \lvert a \rvert)^{\rho_1(h) + \tau}} e^{k \log \lvert \alpha \rvert} \to 0
	\end{equation*}
	as $k \to \infty$, since $\rho_1(h) + \tau < 1$ and $\log\lvert \alpha \rvert < 0$. 
	Thus, there exists $N = N(R)\ge0$ such that for all $k \geq N$
	\begin{equation*}\label{Lemma:Baker:RealPartToInftyVar2:ProofPropN}
		e^{(R + 2k \lvert a \rvert)^{\rho_1(h) + \tau}} e^{k  \log \lvert \alpha \rvert} < \frac{\Re(a)}{4}.
	\end{equation*}
	We now choose any $0 < \varepsilon < \min\lbrace 1,\delta_g \rbrace$ such that
	\begin{equation*}\label{Lemma:Baker:RealPartToInftyVar2:ProofPropEps}
		e^{(R + 2k \lvert a \rvert)^{\rho_1(h) + \tau}} e^{k  \log \lvert \alpha \rvert} \varepsilon < \frac{\Re(a)}{4}
	\end{equation*}
	for all $k \in \lbrace 0, \dots, N-1 \rbrace$.
	Therefore, for any $k \in \mathbb{N}$ we have
	\begin{equation}\label{Lemma:Baker:RealPartToInftyVar2:ProofPropCombined}
		e^{(R + 2k \lvert a \rvert)^{\rho_1(h) + \tau}} e^{k  \log \lvert \alpha \rvert} \varepsilon < \frac{\Re(a)}{4}.
	\end{equation}
	
	We can now prove properties~\eqref{Lemma:Baker:RealPartToInftyVar2:ProofProp1} and \eqref{Lemma:Baker:RealPartToInftyVar2:ProofProp2} by induction over $k$. Take  $z \in U_{r,R}$ and $ w \in \D(0,\varepsilon)$. The case $k=0$ is obvious.
	
	Assume that both properties hold for some $k\in\mathbb{N}$. In particular, we have $\lvert z_k \rvert \geq \Re(z_k) > L > \log \frac{2}{\Re(a)}$ and thus $\left\lvert e^{-z_k} \right\rvert < \frac{\Re(a)}{2} \leq \frac{\lvert a \rvert}{2} $. 
	Observe that for all $z \in \mathbb{C}$ with $\lvert z \rvert \geq L$, and $w \in \D(0,\varepsilon)$
	\begin{equation}
		\label{Lemma:Baker:RealPartToInftyVar2:ErrorEquation}
		\begin{split}
			\lvert h(z,w) \rvert &\leq M(\lvert z \rvert, \delta_g;h) \leq e^{\lvert z \rvert^{\rho_1(h) + \tau}}.
		\end{split}
	\end{equation}
	Then,
	\begin{equation}
		\label{Lemma:Baker:RealPartToInftyVar2:ProofPropUpwardReal}
		\begin{split}
			\lvert z_{k+1} \rvert &= \left\lvert z_k + a + e^{-z_k} + w_k h(z_k, w_k) \right\rvert \\
			&\leq \lvert z_{k} \rvert + \lvert a \rvert + \left\lvert e^{-z_k} \right\rvert + \lvert w_k h(z_k, w_k) \rvert \\
			&\leq R + 2k \lvert a \rvert + \lvert a \rvert +  \frac{\lvert a \rvert}{2} + \lvert w_k \rvert  e^{\lvert z_k \rvert^{\rho_1(h) + \tau}}.
		\end{split}
	\end{equation}
	The induction hypothesis yields $\lvert z_k \rvert < R + 2k \lvert a \rvert$ and therefore,
	\begin{equation*}
		\begin{split}
			e^{\lvert z_k \rvert^{\rho_1(h) + \tau}} &\leq e^{( R + 2k \lvert a \rvert )^{\rho_1(h) + \tau}}.
		\end{split}
	\end{equation*}
	For $w_k$, we have
	\begin{equation*}
		\begin{split}
			\lvert w_k \rvert = \lvert g^k(w) \rvert \leq \lvert \alpha \rvert^k \lvert w \rvert \leq e^{k \log \lvert \alpha \rvert} \varepsilon.
		\end{split}
	\end{equation*}
	Applying these inequalities to \eqref{Lemma:Baker:RealPartToInftyVar2:ProofPropUpwardReal} together with \eqref{Lemma:Baker:RealPartToInftyVar2:ProofPropCombined}, we obtain
	\begin{equation*}
		\begin{split}
			\lvert z_{k+1} \rvert &\leq R + (2k + 1) \lvert a \rvert +  \frac{\lvert a \rvert}{2} + \frac{\Re(a)}{4} \\
			&\leq R + (2k + 1) \lvert a \rvert +  \frac{\lvert a \rvert}{2} + \frac{\lvert a \rvert}{4}\\
			&< R + 2(k+1) \lvert a \rvert.
		\end{split}
	\end{equation*}
	Similarly, we deduce
	\begin{equation*}
		\begin{split}
			\Re(z_{k+1}) &= \Re\left( z_k + a + e^{-z_k} + w_k h(z_k, w_k) \right) \\
			&\geq \Re(z_k) + \Re(a) - \left\lvert e^{-z_k} \right\rvert - \lvert w_k h(z_k, w_k) \rvert \\
			&> r + \frac{k\Re(a)}{4} + \Re(a) - \frac{\Re(a)}{2}  - e^{(R + 2k \lvert a \rvert)^{\rho_1(h) + \tau}} e^{k  \log \lvert \alpha \rvert} \varepsilon\\
			&> r + \frac{k\Re(a)}{4} + \frac{\Re(a)}{2} - \frac{\Re(a)}{4}\\
			&= r + \frac{(k+1)\Re(a)}{4}.
		\end{split}
	\end{equation*}
	This completes the proof for $\rho_1(h) < 1$.
	
	We now assume that $0$ is super-attracting for $g$ and $\rho_1(h) < \infty$. Hence, $g'(0) = \dots = g^{(d-1)}(0) = 0$ and $g^{(d)}(0) \neq 0$ for some $d \geq 2$. 
	As a consequence, there exists $\widetilde{g} : \mathbb{C} \to \mathbb{C}$ such that $g(w) = w^d \widetilde{g}(w)$ for all $w \in \mathbb{C}$. For $\delta_g > 0$ define
	\begin{equation*}
		C_g := \max_{\lvert w \rvert \leq \delta_g} \lvert \widetilde{g}(w) \rvert > 0.
	\end{equation*}
	Up to choosing a smaller $\delta_g$, we may assume that the inequality $\delta_g^d C_g \leq \delta_g$ holds.
    This ensures that for all $w \in D(0,\delta_g)$ and $k\in\mathbb{N}$
	\begin{equation}\label{Lemma:Baker:RealPartToInftyVar2:ProofPropIteratesUnderg}
		\begin{split}
			\lvert w_k \rvert &= \lvert g^k(w) \rvert \\
			&\leq C_g^{1 + d + \dots + d^{k-1}} \lvert w \rvert^{d^k}.
		\end{split}
	\end{equation}
	Set
	\begin{equation*}
		C := \frac{\log C_g}{1-d}.
	\end{equation*}
	Similarly to the case $\rho_1(h)<1$, we choose $L > 0$ such that for all $\sigma \geq L$
	\begin{equation*}
		M(\sigma, \delta_g; h) \leq e^{\sigma^{\rho_1(h) + 1}}.
	\end{equation*}
	Without loss of generality, we may assume $L > \log \frac{2}{\Re(a)}$. Consider any radii satisfying $R>r\ge L$. 
	For $t > 0$ small enough such that
	\begin{equation*}
		\log \frac{1}{t} + C > 0
	\end{equation*}
	we have
	\begin{equation*}
		e^{(R + 2k \lvert a \rvert)^{\rho_1(h) + 1}} e^{C}e^{-d^k (\log \frac{1}{t} + C)} \to 0
	\end{equation*}
	as $k \to \infty$ since $\rho_1(h) + 1 < \infty$.
	This convergence is uniform on $(0,t]$. Thus, there exists $N = N(R,t)\ge0$ such that for all $k \geq N$ and $\varepsilon \in (0,t]$ we have
	\begin{equation*}\label{Lemma:Baker:RealPartToInftyVar2:ProofPropN2}
		e^{(R + 2k \lvert a \rvert)^{\rho_1(h) + 1}} e^{C}e^{-d^k (\log \frac{1}{\varepsilon} + C)} < \frac{\Re(a)}{4}.
	\end{equation*}
	We now choose any $0 < \varepsilon < \min\lbrace t, \delta_g \rbrace$ such that
	\begin{equation*}\label{Lemma:Baker:RealPartToInftyVar2:ProofPropEps2}
		e^{(R + 2k \lvert a \rvert)^{\rho_1(h) + 1}} e^{C - d^k C}\varepsilon^{d^k} =  e^{(R + 2k \lvert a \rvert)^{\rho_1(h) + 1}} e^{C}e^{-d^k (\log \frac{1}{\varepsilon} + C)} < \frac{\Re(a)}{4}
	\end{equation*}
	for all $k \in \lbrace 0, \dots, N-1 \rbrace$. 
	Therefore, for all $k \in \mathbb{N}$ we have
	\begin{equation}\label{Lemma:Baker:RealPartToInftyVar2:ProofPropCombined2}
		e^{(R + 2k \lvert a \rvert)^{\rho_1(h) + 1}} e^{C}e^{-d^k (\log \frac{1}{\varepsilon} + C)} < \frac{\Re(a)}{4}.
	\end{equation}
	
	The rest of the proof now follows as for $\rho_1(h) < 1$ by using induction over $k$ to show properties~\eqref{Lemma:Baker:RealPartToInftyVar2:ProofProp1} and \eqref{Lemma:Baker:RealPartToInftyVar2:ProofProp2}. We just need to observe that in the induction step, we use that for all $z \in U_{r,R}$ and $ w \in \D(0,\varepsilon)$, the inequality \eqref{Lemma:Baker:RealPartToInftyVar2:ErrorEquation} together with \eqref{Lemma:Baker:RealPartToInftyVar2:ProofPropIteratesUnderg} and \eqref{Lemma:Baker:RealPartToInftyVar2:ProofPropCombined2} yields
	\begin{equation*}
		\begin{split}
			\lvert w_k h(z_k, w_k) \rvert &\leq e^{\lvert z_k \rvert^{\rho_1(h) + 1}}  \lvert w_k \rvert \\
			&\leq e^{\lvert z_k \rvert^{\rho_1(h) + 1}}  C_g^{1 + d + \dots + d^{k-1}} \lvert w \rvert^{d^k}\\
            &\leq e^{(R + 2k \lvert a \rvert)^{\rho_1(h) + 1}} e^{\frac{1 - d^k}{1 - d} \log C_g} \varepsilon^{d^k} \\
            &\leq e^{(R + 2k \lvert a \rvert)^{\rho_1(h) + 1}} e^{(1 - d^k) C} e^{d^k \log \varepsilon} \\
			&\leq e^{(R + 2k \lvert a \rvert)^{\rho_1(h) + 1}} e^{C}e^{-d^k (\log \frac{1}{\varepsilon} + C)} \\
			&< \frac{\Re(a)}{4}.
		\end{split}
	\end{equation*}
	for all $k\in\mathbb{N}$.
\end{proof}


\begin{proof}[Proof of Theorem~{\ref{Theorem:BakerProductErrorTEF}}]  
Thanks to Proposition~\ref{lem:absorbing implies everything}, it suffices to prove that there is a bulging absorbing domain. 
	Let $ L > \log \frac{2}{\Re(a)} $ satisfy the conclusions of Lemma~\ref{Lemma:Baker:RealPartToInftyVar2}.
	By Lemma~\ref{lem:basic Baker 1d}, the set $D := \left\lbrace z\in\C\mid \Re(z) > L \right\rbrace$ is an absorbing domain. The sets
	\begin{equation*}
		U_{L, R} := \left\lbrace  u\in \mathbb{C} \mid L < \Re(u), \lvert u \rvert < R \right\rbrace  = \left\lbrace u \in D \mid \lvert u \rvert < R \right\rbrace,
	\end{equation*}
	for $R > L $, form an open covering of $D$. We show that each $U_{L,R}$ bulges. It then follows that $D$ bulges as well.
	
	Given any $R > L$, Lemma~\ref{Lemma:Baker:RealPartToInftyVar2} yields $\varepsilon > 0$ such that
	\begin{equation*}
		\Re(z_k) \geq \frac{k\Re(a)}{4} + L
	\end{equation*}
	for all $(z,w) \in U_{L,R} \times \D(0,\varepsilon)$ and $k\in\mathbb{N}$. Thus, using $\Re(a) > 0$, we obtain
	\begin{equation*}
		\lvert z_k \rvert \geq \lvert \Re(z_k) \rvert \geq \frac{k \Re(a)}{4} + L \to \infty
	\end{equation*}
	as $k\to\infty$. Since the lower bound does not depend on $z$, we have uniform convergence towards $\infty$ in the $z$-coordinate.
	Moreover, without loss of generality, we may assume that $\varepsilon > 0$ is small enough such that $g^k(w)$ converges uniformly to $0$ as $k\to\infty$ for all $w \in \D(0,\varepsilon)$.
	
	Taken together, we have uniform convergence of the iterates on $U_{L,R} \times \D(0,\varepsilon)$ to $[1:0:0]$. In particular, this set is contained in a Fatou component of the skew product $F$, i.e., $U_{L,R}$ bulges.
\end{proof}

\section{Non-bulging Baker domains}\label{non-bulging}
  In this section we prove the following result.
  
\begin{thm}\label{THEOREM:NONBULG:BAKER}
	Let $g\colon\C\to\C$ be a non-constant holomorphic function with an attracting fixed point at $0$.
	Then there exists an entire function $h : \mathbb{C} \to \mathbb{C}$ such that the skew product $F_h \colon \mathbb{C}^2 \to \mathbb{C}^2$,
	\begin{equation*}
		 F_h(z,w) = (z + 1 + e^z + wh(z), g(w))
	\end{equation*}
	has a non-bulging Baker domain on the attracting invariant fiber $\{w = 0\}$.
\end{thm}

We will need the following result and the immediate corollary.

\begin{lem}[Runge's Theorem, see for example {\cite[p. 459]{Eremenko_Ljubich_1987}}]\label{Lemma:Runge:Gen}
    Let $K \subset \C$ be compact with connected complement. Then any function holomorphic  in a neighbourhood of $K$ can be uniformly approximated by polynomials.
\end{lem}

\begin{cor}\label{Lemma:Runge}
    Let $E_0, \dots, E_n$ be simply connected, pairwise disjoint, compact subsets of $\C$, let $\psi$ be holomorphic in a neighborhood of  $E = \bigcup_{k=0}^{n} E_k$, and let $\epsilon> 0$ be a real number.
    Then there exists an entire function $h$ such that
	\begin{equation*}
		|h(z) - \psi(z)| < \epsilon
	\end{equation*}
	for all $z \in E$.
\end{cor}

\begin{proof}[Proof of Theorem~\ref{THEOREM:NONBULG:BAKER}]
	In this proof, we construct a skew product with a non-bulging Baker domain on an attracting fiber using Runge's Theorem iteratively. 

    Let $g : \C \to \C$ be given as in the statement, that is, $g$ non-constant holomorphic, $g(0) = 0$, and $|{g'(0)}| < 1$. 
    
    Let $0 < \delta_g < 1$ be such that the disk $\D(0,\delta_g)$ lies entirely within the immediate attracting basin of $g$ at $0$, i.e., $g^n(w) \to 0$ as $n\to \infty$ for all $w\in \D(0,\delta_g)$, and such that
    \begin{equation}\label{PROOF:NONBULG:BAKER:DefDeltaG}
        |{g(w)}| \leq |{w}|.
    \end{equation}
    In particular, $g(\D(0,\delta_g)) \subset \D(0,\delta_g)$.

    We are interested in skew products $F_h\colon\C^2 \to \C^2$ of the form
    \begin{equation}\label{PROOF:NONBULG:BAKER:DefF}
         F_h(z,w) = (f(z) + wh(z), g(w))
    \end{equation}
    where $h : \C \to \C$ is a suitably chosen holomorphic function, to be constructed later.
    On the invariant fiber, we have the Fatou's function 
    $$
    f(z)=  z + 1 + e^{-z}.
    $$ 
    From one-dimensional dynamics, as described in Lemma \ref{LEMMA:OneDim}, we know that $f$ has an invariant Baker domain $B$ containing the absorbing domain $\{z \in \C\mid\Re(z) > \log(2)\}$, and this is independent of the choice of $h$.

        The key idea is to construct $h$ in such a way, that there exist points arbitrarily close to some element in $B \times \{0\}$ which, under arbitrarily many iterations, are mapped back to a bounded neighbourhood of the origin. The details of this behaviour, as well as the construction of $h$, are provided in the following result.

    \begin{lem}\label{PROOF:NONBULG:BAKER:SUBLEMMA}
        There exists an entire function $h : \C \to \C$, a point $x_0 \in B$, and a sequence of non-zero complex numbers $(w_n)_{n \in \N}$ such that $\lim_{n\to\infty}w_n = 0$, $|{w_n}| < \delta_g$ for all $n\in\N$, and the skew product $F_h$ as in \eqref{PROOF:NONBULG:BAKER:DefF} satisfies
        \begin{equation}\label{PROOF:NONBULG:BAKER:SUBLEMMA:EQ1}
            |{\pi_z (F_h^{n+1}(x_0, w_n)})| \leq 3
        \end{equation}
        for all $n\in\N$, where $\pi_z$ denotes the projection onto the first coordinate.
    \end{lem}

\noindent\textit{Proof of Theorem~\ref{THEOREM:NONBULG:BAKER} (continued).}
	Let $h$, $x_0$, and $(w_n)_{n\in\N}$ be given by Lemma \ref{PROOF:NONBULG:BAKER:SUBLEMMA}.
    We proceed by contradiction to show that the Baker domain $B \subset \C$ of $f$ does not bulge for $F_h$. 

    We must consider the normality of the iterates of $F_h$ with respect to the compactification $\P^2(\C)$. 
    By Lemma \ref{Lemma:NormalityRelationsBakerVariant}, bulging with respect to this notion of normality implies bulging with respect to the one-point compactification $\widehat{\C^2}$.
    Hence, it suffices to prove the statement for $\widehat{\C^2}$-normality.

    Assume by contradiction that $B$ bulges to a Fatou component $\Omega \subset \C^2$ of $F_h$, i.e., $B \times \{0\} \subset \Omega$.
    Since $x_0 \in B$, the definition of normality yields an open set $V \subset \Omega$ with $(x_0, 0) \in V$ and a sequence of iterates $(F_h^{n_k})_{k\in\N}$ such that either there exists a holomorphic function $\Phi : V \to \C^2$ with
    \begin{equation}\label{PROOF:NONBULG:BAKER:ToLimitFct}
        F_h^{n_k} \to \Phi
    \end{equation}
    uniformly on $V$ as $k \to \infty$, or
    \begin{equation}\label{PROOF:NONBULG:BAKER:ToInfActual}
        F_h^{n_k} \to \infty
    \end{equation}
    uniformly on $V$ as $k \to \infty$.

    On the fiber itself, we have
    \begin{equation*}
        \pi_z (F_h^{n_k}(x_0, 0)) = f^{n_k}(x_0) \to \infty
    \end{equation*}
    as $k \to \infty$, since $x_0 \in B$ by Lemma \ref{PROOF:NONBULG:BAKER:SUBLEMMA}. 
    In particular, there cannot be a limit function $\Phi \in \Hol(V, \C^2)$ satisfying \eqref{PROOF:NONBULG:BAKER:ToLimitFct}.

    Therefore, we obtain uniform divergence to infinity in the form of \eqref{PROOF:NONBULG:BAKER:ToInfActual}, i.e., for every $M > 0$ there exists $K \in \N$ such that for every $k \geq K$ and $(z,w) \in V$
    \begin{equation}\label{PROOF:NONBULG:BAKER:ToInf}
        \|{F_h^{n_k}(z,w)}\| \geq M.
    \end{equation}
    However, we will now show that there are points in $V$ which do not satisfy this property for $M = 4$ and the corresponding constant $K$. 
    Since $V$ is open, $(x_0,0) \in V$, and $\lim_{n\to\infty} w_n = 0$, there exists $\tilde{K} \in \N$ such that $(x_0, w_{n_k - 1}) \in V$ for all $k \geq \tilde{K}$ where the points $w_n$ are those given by Lemma \ref{PROOF:NONBULG:BAKER:SUBLEMMA}.
    Furthermore, by \eqref{PROOF:NONBULG:BAKER:SUBLEMMA:EQ1}, we have
    \begin{equation}
        |{\pi_z( F_h^{n_k}(x_0, w_{n_k - 1})})| \leq 3
    \end{equation} 
    for all $k\in\N$.
    Combining these results with \eqref{PROOF:NONBULG:BAKER:DefDeltaG}, we obtain for $k \geq \max\{K, \tilde{K}\}$
    \begin{align*}
        \|{F_h^{n_k}(x_0,w_{n_k - 1})}\| &\leq |{\pi_z( F_h^{n_k}(x_0, w_{n_k - 1})})| + |{g^{n_k}(w_{n_k - 1})}| \\
        &\leq 3 + |{w_{n_k - 1}}| 
        \leq 3 + \delta_g \\
        &< 4 = M
    \end{align*}
    but also $(x_0, w_{n_k - 1}) \in V$, contradicting \eqref{PROOF:NONBULG:BAKER:ToInf}. 
    Hence, the assumption that $B$ bulges is false, and the proof is complete.
\end{proof}

The rest of this section is devoted to the proof of the key Lemma~\ref{PROOF:NONBULG:BAKER:SUBLEMMA}.
    \begin{proof}[Proof of Lemma~\ref{PROOF:NONBULG:BAKER:SUBLEMMA}]
        First, fix a value $0 < \delta \leq \frac{1}{10}$ for the remainder of the proof, and choose the corresponding $x_0$ given by Lemma \ref{LEMMA:OneDim}. In particular, we know that $x_0 \in B$.

        The proof then proceeds in two steps. First, we will recursively define functions $h_k$ such that $F_{h_k}$ satisfies a property like \eqref{PROOF:NONBULG:BAKER:SUBLEMMA:EQ1} for all $0 \leq n \leq k$ and then, in the second step, we will show that this sequence of functions has a limit $h$ which inherits this property.

        To control the behaviour close to the orbit $x_k = f^k(x_0)$, we first define the closed disks 
        \begin{equation*}
            D_k := \overline{\D(x_0 + k, 4\delta)}
        \end{equation*}
        centered at $x_0 + k$, and the sets 
        \begin{equation*}
            H_{k} := \left\{z \in \C\mid-k \leq \Re(z) \leq x_0 + k - \frac{1}{2}, |{\Im(z)}| \leq k+1\right\}.
        \end{equation*}
        These sets are shown in figure~\ref{fig:tikz-simplifiedBakerProof}. Note that
        \begin{equation*}
            \D(x_k, \delta) \subset D_k
        \end{equation*}
        by Lemma \ref{LEMMA:OneDim}. Furthermore, observe that the sets $D_k$ and $H_k$ are all simply connected, compact subsets of $\C$. Furthermore, the sets $D_k$ are pairwise disjoint, and satisfy $D_j \subset H_k$ for all $j < k$, while we have $D_j \cap H_k = \emptyset$ for all $j \geq k$. Finally, the sets $H_k$ form a nested covering of $\C$, that is, $H_k \subset H_{k+1}$ for all $k\in\N$, and $\cup_{k\in\N}H_k = \C$. 

        \begin{figure}[htbp]
            \centering
            \begin{tikzpicture}
                \fill[fill=gray!40]
                    (-4,2) -- (1.25,2) -- (1.25, -2) -- (-4, -2);
                \draw[-]
                    (-4,2) -- (1.25,2) -- (1.25, -2) -- (-4, -2);
                \node[above, text=black!20!gray] at (0.95, 1.3) {$H_2$};
                    
                \filldraw[fill=gray!20, draw=black]
                    (-2.5,0) circle (1);
                    
                \filldraw[fill=gray!20, draw=black]
                    (0,0) circle (1);

                \filldraw[fill=gray!20, draw=black]
                    (2.5,0) circle (1);

                \node[above right] at (-2.5,1) {$D_0$};
                \node[above right] at (0,1) {$D_1$};
                \node[above right] at (2.5,1) {$D_2$};

                \draw[-]
                    (-2.5, -0.1) -- (-2.5,0.1);
                \node[below] at (-2.5,-0.1) {$x_0$};
                \draw[-]
                    (0, -0.1) -- (0,0.1);
                \node[below] at (0,-0.1) {$x_0 + 1$};
                \draw[-]
                    (2.5, -0.1) -- (2.5,0.1);
                \node[below] at (2.5,-0.1) {$x_0 + 2$};

                \draw[-]
                    (-2.5,0) -- ({-2.5 + cos(120)}, {sin(120)});
                \node[right] at ({-2.5 + 0.5 * cos(120)}, {0.5 * sin(120)}) {\small$4\delta$};

                \draw[->]
                    (-4,0) -- (4,0);
                \node[below right] at (4,0) {$\Re$};
            \end{tikzpicture}
            \caption{The first disks $D_k$ together with $H_2$.}
            \label{fig:tikz-simplifiedBakerProof}
        \end{figure}
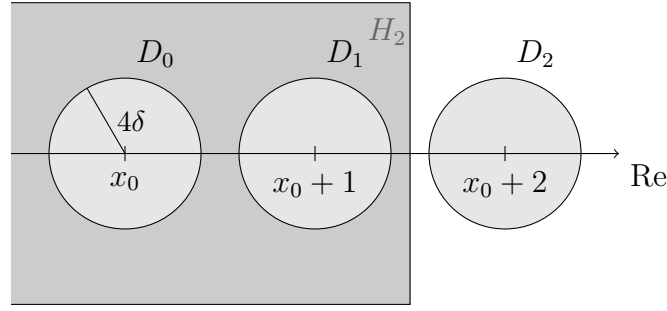
        We will obtain $h$ as limit of the entire functions constructed in the following lemma, using Runge's Theorem.

        \begin{lem}\label{PROOF:NONBULG:BAKER:SUBLEMMA2}
            There exists a sequence of entire functions $(h_k)_{k\in\N}$, along with a sequence $(w_k)_{k\in\N}$ of non-zero complex numbers and a sequence $(\delta_k)_{k\in\N}$ of positive real numbers such that for all integers $k \geq 0$
            \begin{enumerate}
                \item $0 < |{w_k}| < \min\{\frac{1}{k + 1}, \delta_g\}$, \label{PROOF:NONBULG:BAKER:SUBLEMMA:CONSTR1}
                \item $0 < \delta_k < 1$, \label{PROOF:NONBULG:BAKER:SUBLEMMA:CONSTR2}
                \item $\pi_z (F_{h_k}^k(x_0, w_k)) \in \D(x_k, \delta) \subset D_k$, \label{PROOF:NONBULG:BAKER:SUBLEMMA:CONSTR3}
                \item $\|{h_k + \frac{x_{k+1}}{g^k(w_k)}}\|_{\infty, D_{k}} = \sup_{z \in D_k}|{h_k(z) + \frac{x_{k+1}}{g^k(w_k)}}| < 1$, \label{PROOF:NONBULG:BAKER:SUBLEMMA:CONSTR4}
            \end{enumerate}
            and for all integers $k \ge1$
            \begin{enumerate}
                \setcounter{enumi}{4}
                \item $\|{h_{k} - h_{k-1}}\|_{\infty, H_{k}} < \frac{1}{2^{k+1}} \min\{\delta_j\mid j = 0, \dots, k\}$, and \label{PROOF:NONBULG:BAKER:SUBLEMMA:CONSTR5}
                \item for any $\widetilde{h} : \C \to \C$ holomorphic, the inequality
                \begin{equation}
                    \|{\widetilde{h} - h_{k - 1}}\|_{\infty, H_k} < \delta_k
                \end{equation}
                implies
                \begin{equation}
                    |\pi_z (F_{h_{k-1}}^k(x_0, w_k)) - \pi_z( F_{\widetilde{h}}^k(x_0, w_k))| < \delta.
                \end{equation} \label{PROOF:NONBULG:BAKER:SUBLEMMA:CONSTR6} 
            \end{enumerate}
        \end{lem}


        \begin{proof}[Proof of Lemma~\ref{PROOF:NONBULG:BAKER:SUBLEMMA2}]
            Before turning to the construction itself, we observe that
            \begin{equation}\label{EQ:ZeroSpecialCase}
                \pi_z( F^{j}_{\widetilde{h}}(x_0, 0)) = \pi_z (f^{j}(x_0),0) = f^{j}(x_0) = x_{j}
            \end{equation}
            holds for all $j \in \N$ and $\tilde{h} : \C \to \C$ entire.

            For $k=0$, we verify the first four properties, and for $k = 1$, we verify all six properties in order to start the iterative construction.

            For the first step of the construction, let $w_0$ be any complex number in $\D(0,\delta_g)\setminus\{0\}$ and define 
            \begin{equation*}
                h_0 : \C \to \C, z \mapsto \frac{-x_1}{w_0}.
            \end{equation*} 
            Moreover, let $\delta_0 := \frac{1}{2}$. The exact choice of $\delta_0$ is immaterial for the rest of the proof. Only the first four properties have to be proven for $k = 0$, and they follow directly from our definition of $h_0$ and $w_0$. 

            Now, for $k = 1$, using \eqref{EQ:ZeroSpecialCase} and the continuity of both $w \mapsto \pi_z (F^0_{h_{0}}(x_0,w))$ and $w \mapsto \pi_z (F^1_{h_{0}}(x_0,w))$ at $0$, we can choose $w_1 \in \D({0,\min\{\delta_g, \frac{1}{2}\}}) \setminus \{0\}$, which yields property~\eqref{PROOF:NONBULG:BAKER:SUBLEMMA:CONSTR1}, such that
            \begin{equation*}
                \pi_z (F^0_{h_{0}}(x_0,w_1)) \in \D(x_0,\delta)
            \end{equation*}
            and
            \begin{equation*}
                \pi_z (F^1_{h_{0}}(x_0,w_1)) \in \D(x_1,\delta).
            \end{equation*}
            By continuity, we can choose $0 < \delta_1 < 1$ small enough such that for any holomorphic function $\widetilde{h} : \C \to \C$ the inequality
            \begin{equation*}
                \|{\widetilde{h} - h_{0}}\|_{\infty, H_{1}} < \delta_{1}
            \end{equation*}
            implies
            \begin{equation*}
                |{\pi_z( F^{1}_{h_{0}}(x_0, w_{1})) - \pi_z (F^{1}_{\widetilde{h}}(x_0, w_{1})})| < \delta
            \end{equation*}
            and
            \begin{equation}\label{EQ:SecondDeltaCrit}
                \pi_z( F^{1}_{\widetilde{h}}(x_0, w_{1})) \in \D(x_{1}, \delta).
            \end{equation}
            This yields properties \eqref{PROOF:NONBULG:BAKER:SUBLEMMA:CONSTR2} and \eqref{PROOF:NONBULG:BAKER:SUBLEMMA:CONSTR6}.
            Taking $\varepsilon_1 := \frac{1}{4} \min\{\delta_0, \delta_1\} < 1$ and defining
            \begin{equation*}
                \phi_{1} : H_{1} \cup D_{1} \to \C, z \mapsto \begin{cases}
                    h_{0}(z),\;\text{ for}\; z\in H_{1},\\
                    -\frac{x_{2}}{g^{1}(w_{1})},\;\text{ for}\;z \in D_{1},
                \end{cases}
            \end{equation*}
            we can apply Runge's Theorem Lemma \ref{Lemma:Runge} to get a polynomial $h_1$ satisfying 
            \begin{equation*}
                \|{h_1 - \phi_1}\|_{\infty, H_1 \cup D_1} < \varepsilon_1.
            \end{equation*}
            By construction, we obtain properties \eqref{PROOF:NONBULG:BAKER:SUBLEMMA:CONSTR4} and \eqref{PROOF:NONBULG:BAKER:SUBLEMMA:CONSTR5}. Property~\eqref{PROOF:NONBULG:BAKER:SUBLEMMA:CONSTR3} follows from \eqref{EQ:SecondDeltaCrit}. This concludes the first step of our inductive construction.

            Now, given $h_0, \dots, h_{k-1}$, $w_0, \dots, w_{k-1}$ and $\delta_0, \dots, \delta_{k-1}$ constructed according to properties \eqref{PROOF:NONBULG:BAKER:SUBLEMMA:CONSTR1} to \eqref{PROOF:NONBULG:BAKER:SUBLEMMA:CONSTR6}, we define the next elements of our sequences. 

            As before, \eqref{EQ:ZeroSpecialCase} and the continuity of $w \mapsto \pi_z( F^j_{h_{k-1}}(x_0,w) )$ in $0$ for $j = 0, \dots, k$ allow us to choose $w_{k} \in \D(0,\delta_g) \setminus \{0\}$ small enough such that $|{w_{k}}| < \frac{1}{k+1}$, which is required for property~\eqref{PROOF:NONBULG:BAKER:SUBLEMMA:CONSTR1}, and 
            \begin{equation}\label{PROOF:NONBULG:BAKER:SUBLEMMA:CONSTR:ST1}
                \pi_z( F^{j}_{h_{k-1}}(x_0, w_{k}) )\in \D(x_{j}, \delta) \subset D_j
            \end{equation}
            for all $0 \leq j \leq k$.

            From \eqref{PROOF:NONBULG:BAKER:SUBLEMMA:CONSTR:ST1} we can, in particular, derive that the $z$-coordinates of the first $k-1$ images of $(x_0, w_{k})$ under $F_{h_{k-1}}$ remain in $ D_0 \cup \dots \cup D_{k-1} \subset H_{k}$. Consequently, by continuity, if an entire function $\widetilde{h}$ is close enough to $h_{k-1}$ on $H_k$, then the first $k$ preimages of $(x_0, w_{k})$ under $F_{\widetilde{h}}$ will stay arbitrarily close to those under $F_{h_{k-1}}$.

            Thus, we find $ 0 < \delta_{k} < 1 $ small enough such that for any holomorphic functions $\widetilde{h} : \C \to \C$ the inequality
            \begin{equation}\label{PROOF:NONBULG:BAKER:SUBLEMMA2:EQ12}
                \|{\widetilde{h} - h_{k-1}}\|_{\infty, H_{k}} < \delta_{k}
            \end{equation}
            implies
            \begin{equation}\label{PROOF:NONBULG:BAKER:SUBLEMMA2:EQ13}
                |{\pi_z( F^{k}_{h_{k-1}}(x_0, w_{k})) - \pi_z (F^{k}_{\widetilde{h}}(x_0, w_{k})})| < \delta
            \end{equation}
            and
            \begin{equation}\label{PROOF:NONBULG:BAKER:SUBLEMMA2:EQ14}
                \pi_z( F^{k}_{\widetilde{h}}(x_0, w_{k})) \in \D(x_{j}, \delta).
            \end{equation}
            This yields properties \eqref{PROOF:NONBULG:BAKER:SUBLEMMA:CONSTR2} and \eqref{PROOF:NONBULG:BAKER:SUBLEMMA:CONSTR6}.

            Let $\varepsilon_{k} := \frac{1}{2^{k+1}}\min\{\delta_j\mid 0 \leq j \leq k\}$.
            Consider the holomorphic function
            \begin{equation}\label{PROOF:NONBULG:BAKER:DEFPHI}
                \phi_{k} : H_{k} \cup D_{k} \to \C, z \mapsto \begin{cases}
                    h_{k-1}(z),\;\text{ for}\; z\in H_{k},\\
                    -\frac{x_{k+1}}{g^{k}(w_{k})},\;\text{ for}\;z \in D_{k}.
                \end{cases}
            \end{equation}
            Runge's Theorem \ref{Lemma:Runge} applied to the two sets $H_k$ and $D_k$ yields a polynomial $h_{k}$ which coincides with $\phi_{k}$ on its domain of definition $H_{k} \cup D_{k}$ up to the error $\varepsilon_{k}$. Hence, properties~\eqref{PROOF:NONBULG:BAKER:SUBLEMMA:CONSTR4} and \eqref{PROOF:NONBULG:BAKER:SUBLEMMA:CONSTR5} hold for $k$. 
            
            Moreover, property~\eqref{PROOF:NONBULG:BAKER:SUBLEMMA:CONSTR3} holds for $k$, since $h_k$ satisfies \eqref{PROOF:NONBULG:BAKER:SUBLEMMA2:EQ12} by construction, which in turn implies \eqref{PROOF:NONBULG:BAKER:SUBLEMMA2:EQ14}.  

            This completes the construction step.
        \end{proof}

\noindent\textit{Proof of Lemma~\ref{PROOF:NONBULG:BAKER:SUBLEMMA} (continued).}

        Given the construction from Lemma \ref{PROOF:NONBULG:BAKER:SUBLEMMA2}, we now show the existence of a limit function $h$ which satisfies the property required by Lemma \ref{PROOF:NONBULG:BAKER:SUBLEMMA}.

        First observe, that property~\eqref{PROOF:NONBULG:BAKER:SUBLEMMA:CONSTR1} of Lemma~\ref{PROOF:NONBULG:BAKER:SUBLEMMA2} directly implies $w_k \to 0$ as $k \to \infty$. 
        By writing $h_k$ as a telescopic sum, that is,
        \begin{equation*}
            h_k = \sum_{j = m}^{k - 1} (h_{j + 1} - h_{j}) + h_m = h_m +  \sum_{j = m}^{k-1} \left(h_{j + 1} - h_{j}\right)
        \end{equation*}
        for $k \geq m$, and observing that properties~\eqref{PROOF:NONBULG:BAKER:SUBLEMMA:CONSTR2} and \eqref{PROOF:NONBULG:BAKER:SUBLEMMA:CONSTR5} in Lemma~\ref{PROOF:NONBULG:BAKER:SUBLEMMA2} yield
        \begin{equation*}
            \sum_{j = m}^{\infty} \|{h_{j + 1} - h_{j}}\|_{\infty, H_m} < \sum_{j = m}^{\infty} \|{h_{j + 1} - h_{j}}\|_{\infty, H_{j+1}} < \sum_{j = m}^{\infty} \frac{1}{2^{j+2}} < \infty,
        \end{equation*}
        we see that the telescopic sum forms a Cauchy sequence on $H_m$ and therefore has a uniform limit function 
        \begin{equation*}
            \lim_{k \to \infty} h_k = h_m + \sum_{j = m}^{\infty} (h_{j + 1} - h_{j})
        \end{equation*}
        on $H_m$. 
        Using the fact that $\bigcup_{k\in\N}H_k = \C$, $H_k \subset H_{k+1}$, and Weierstra\ss~Theorem, we obtain an entire limit function $h = \lim_{k\to \infty}h_k$ on $\C$. 

        Moreover, using the same argument as before, this limit function satisfies
        \begin{equation}\label{PROOF:NONBULG:BAKER:EQ2}
            \begin{split}
                \|{h - h_n}\|_{\infty, H_{n+1}} &\leq \ \sum_{k = n}^{\infty} \ \|{h_{k + 1} - h_{k}}\|_{\infty, H_{n+1}} \\
                &\leq \sum_{k = n+1}^{\infty} \frac{1}{2^{k+1}} \min\{\delta_j\mid j = 0, \dots, k\} \\
                &\leq \sum_{k = n+1}^{\infty} \frac{1}{2^{k+1}} \delta_{n+1} < \delta_{n+1}.
            \end{split}
        \end{equation}

        We now verify that our construction of $h$ ensures that the map
        \begin{equation*}
            F_h : \C^2 \to \C^2, (z,w) \mapsto  (f(z) + wh(z), g(w))
        \end{equation*}
        satisfies the relation \eqref{PROOF:NONBULG:BAKER:SUBLEMMA:EQ1}. 
       Let $n\in\N$. We first show that $h$ is sufficiently close to the recursively defined function $h_n$ to satisfy $ \pi_z( F_h^{n}(x_0, w_n) )\in D_n$. By using \eqref{PROOF:NONBULG:BAKER:EQ2} for $n-1$, we can apply property~\eqref{PROOF:NONBULG:BAKER:SUBLEMMA:CONSTR6} of Lemma~\ref{PROOF:NONBULG:BAKER:SUBLEMMA2} to $h$ and obtain
        \begin{equation*}
            |{\pi_z( F_{h_{n-1}}^n(x_0, w_n)) - \pi_z (F_{h}^n(x_0, w_n)})| < \delta.
        \end{equation*}
        Similarly, we can apply property~\eqref{PROOF:NONBULG:BAKER:SUBLEMMA:CONSTR6} of Lemma~\ref{PROOF:NONBULG:BAKER:SUBLEMMA2} to $h_n$ and obtain
        \begin{equation*}
            |{\pi_z( F_{h_{n-1}}^n(x_0, w_n)) - \pi_z (F_{h_n}^n(x_0, w_n)})| < \delta.
        \end{equation*}
        Furthermore, property~\eqref{PROOF:NONBULG:BAKER:SUBLEMMA:CONSTR3} of Lemma~\ref{PROOF:NONBULG:BAKER:SUBLEMMA2} yields
        \begin{equation*}
            \pi_z( F_{h_n}^n(x_0, w_n)) \in \D(x_n, \delta),
        \end{equation*}
        that is
        \begin{equation*}
            |{x_n - \pi_z( F_{h_n}^n(x_0, w_n)})| < \delta.
        \end{equation*}
        And lastly, by \eqref{LEMMA:OneDim:EQ1} from Lemma \ref{LEMMA:OneDim}, we have
        \begin{equation*}
            |{x_n - (x_0 + n)}| < \delta.
        \end{equation*}
        Together, we can conclude
        \begin{equation}\label{PROOF:NONBULG:BAKER:EQ1}
            \pi_z( F_h^{n}(x_0, w_n) )\in \overline{\D(x_0 + n, 4\delta)} = D_n.
        \end{equation}

        Now, we can write
        \begin{equation}
            \label{PROOF:NONBULG:BAKER:CONC_EQ1}
            \begin{split}
                |{\pi_z (F_h^{n+1}(x_0, w_n)})| &= |{f(\pi_z (F_h^{n}(x_0, w_n))) + g^n(w_n)h(\pi_z (F_h^{n}(x_0, w_n)))}| \\
                &\leq |f(\pi_z (F_h^{n}(x_0, w_n))) - x_{n+1}| \\
                & \qquad + |{g^n(w_n)h(\pi_z (F_h^{n}(x_0, w_n))) + x_{n+1}}| \\
                &\leq |{f(\pi_z (F_h^{n}(x_0, w_n)))- x_{n+1}}|\\
                & \qquad + |{g^n(w_n)}|\left|{h(\pi_z (F_h^{n}(x_0, w_n))) + \frac{x_{n+1}}{g^n(w_n)}}\right|\\
                &\leq |{f(\pi_z( F_h^{n}(x_0, w_n))) - x_{n+1}}| + \left\|{h +  \frac{x_{n+1}}{g^n(w_n)}}\right\|_{\infty, D_{n}}
            \end{split}
        \end{equation}
        where we used \eqref{PROOF:NONBULG:BAKER:EQ1} in the last step.

        For the first term in \eqref{PROOF:NONBULG:BAKER:CONC_EQ1}, we use \eqref{LEMMA:OneDim:EQ2} and \eqref{PROOF:NONBULG:BAKER:EQ1} to obtain
        \begin{equation}
            \label{PROOF:NONBULG:BAKER:CONC_EQ2}
            \begin{split}
                |{f(\pi_z (F_h^{n}(x_0, w_n))) - x_{n+1}}| &= |{f(\pi_z (F_h^{n}(x_0, w_n))) - f(x_n)}|\\
                &< 10 \delta \leq 1.
            \end{split}
        \end{equation}
        Applying property~\eqref{PROOF:NONBULG:BAKER:SUBLEMMA:CONSTR4} of Lemma~\ref{PROOF:NONBULG:BAKER:SUBLEMMA2} and inequality \eqref{PROOF:NONBULG:BAKER:EQ2} to the second term in \eqref{PROOF:NONBULG:BAKER:CONC_EQ1}, we arrive at
        \begin{equation}
            \label{PROOF:NONBULG:BAKER:CONC_EQ3}
            \begin{split}
                \left\|{h +  \frac{x_{n+1}}{g^n(w_n)}}\right\|_{\infty, D_{n}} &\leq \|{h - h_{n}}\|_{\infty, D_{n}} + \left\|{h_n + \frac{x_{n+1}}{g^n(w_n)}}\right\|_{\infty, D_{n}} \\
                &\leq \|{h - h_{n}}\|_{\infty, H_{n+1}} + 1
                \leq \delta_{n+1} + 1 < 2.
            \end{split}
        \end{equation}
        Combining \eqref{PROOF:NONBULG:BAKER:CONC_EQ1}, \eqref{PROOF:NONBULG:BAKER:CONC_EQ2}, and \eqref{PROOF:NONBULG:BAKER:CONC_EQ3}, we obtain
        \begin{equation*}
            |{\pi_z F_h^{n+1}(x_0, w_n)}| \leq 3,
        \end{equation*} 
        which is \eqref{PROOF:NONBULG:BAKER:SUBLEMMA:EQ1}. This concludes the proof.
    \end{proof}

\begin{rem} The previous proof can be easily adapted to the case of Fatou's functions of the form $f(z)=  z + a + e^{-z}$ with $a\in(0,+\infty)$. 
\end{rem}

\bibliographystyle{amsalpha}
\bibliography{bulging}
\end{document}